\documentclass{amsart}
\usepackage{amsfonts,amssymb,amsmath,amsthm}
\usepackage{url}
\usepackage{enumerate}

\newtheorem{thrm}{Theorem}[section]
\newtheorem{lem}[thrm]{Lemma}
\newtheorem{prop}[thrm]{Proposition}
\newtheorem{cor}[thrm]{Corollary}
\theoremstyle{definition}
\newtheorem{definition}[thrm]{Definition}
\newtheorem{remark}[thrm]{Remark}
\newtheorem{example}[thrm]{Example}
\numberwithin{equation}{section}

\DeclareMathOperator{\sym}{Sym} \DeclareMathOperator{\Inv}{Inv}
\DeclareMathOperator{\cay}{Cay} \DeclareMathOperator{\aut}{Aut}
\DeclareMathOperator{\iso}{Iso}  \DeclareMathOperator{\fis}{Fis}
\DeclareMathOperator{\im}{Im}
\makeatletter
\renewcommand{\subsection}{\@startsection{subsection}{2}{0mm}{-2mm}{-2mm}
{\bf\normalsize}}

\author{Fatemeh Raei Barandagh}
\address{Department of Mathematics,  K. N. Toosi University of Technology,
 P. O. Box: 16315-1618, Tehran, Iran.
}
\email{f.raei@dena.kntu.ac.ir}
\author{Amir Rahnamai Barghi}
\address{
School of Mathematics and Statistics, Carleton University, Ottawa, ON, K1S 5B6\\
}
\email{ rahnama@math.carleton.ca}

\begin{document}

\title[On  circular-arc graphs with association schemes]{On  circular-arc graphs with association schemes}

\begin{abstract}
In this paper, we give a characterization of the class of all
circular-arc graphs whose schemes are association. Moreover, all
association schemes which are the scheme of a circular-arc graph
are characterized, specially it is proved that they are Schurian.
\end{abstract}
\maketitle

\section{Introduction} \label{sect1}

In \cite{Wax},  B. Weisfeiler and A. Leman  have shown
that  a special matrix algebra is assigned to a given graph which
contains the adjacency matrix of the graph. In fact, this algebra
is the adjacency algebra of a scheme. The {\it scheme of a graph}
is the smallest scheme on the vertex set of the graph such that
the edge set of which is the union of some basic relations of the
scheme.
 The scheme of forests, interval graphs and  some special classes of graphs have been studied in~\cite{EPT}.
In this paper,  we study the scheme of a circular-arc graph which
is the intersection graph of a family  of arcs of a circle.

Circular-arc graphs received considerable attention since a
series of papers by Tucker in \cite{T1,T2,T3}, and  by
Dur$\acute{a}$n, Lin and McConnel  in \cite{D, LS, M}. Various
subclasses of circular-arc graphs have been also studied. Among
these are the proper circular-arc graphs, unit circular-arc
graphs, Helly circular-arc graphs and co-bipartite circular-arc
graphs. Several characterizations and recognition algorithms have
been formulated for circular-arc graphs and its subclasses.  But
in this paper, we correspond a finite or algebraic  description
for circular arc graphs instead of description according arcs of a
circle.

 Let
$n$ be a positive integer and let $S$ be a subset  of
$\mathbb{Z}_n$  such that $S=\{\pm 1,\ldots, \pm k\}$ for $0\leq
2k+1< n$. Then the Cayley graph $\cay(\mathbb{Z}_n,S)$ is a
circular-arc graph (see Sec. \ref{cmn graph}), and  we call it
{\it elementary circular-arc graph}. For $k=0$ it is empty graph
and for $k=1$ it is an undirected  cycle.

We say that $\Gamma$ is a {\it graph with association scheme} if
the scheme of $\Gamma$ is association.
  Our main results give a characterization
of    circular-arc graphs with  association schemes in the terms
of lexicographic product of graphs.

In graph theory,
  what we have called the {\it lexicographic product} or {\it composition} of graphs is also often called
the {\it wreath product}. The term wreath product comes from
group theory, and it is also defined in scheme theory.

In fact, analysis  duplicating the vertices of an elementary
circular-arc graph shows that  the lexicographic product  of an
elementary circular-arc graph and a complete graph is a
circular-arc graph.

The rest of this section is to state our results. The following
theorem provides a necessary and sufficient condition for a
circular-arc graph whose scheme is  association.

\begin{thrm}\label{graphtheorem}
A graph is a circular-arc graph  with association scheme if and
only if it is isomorphic to the lexicographic product of  an
elementary circular-arc graph and a complete graph.
\end{thrm}

One can associate to any finite permutation group $G$ a scheme,
denoted by $\Inv(G)$. The scheme associated to the dihedral
group~$D_{2n}$ is called {\it dihedral} scheme.

The class of {\it forestal schemes} have been defined
         inductively  by means of  direct sums and wreath products  in~\cite{EPT}. The scheme of
cographs, trees and interval graphs are forestal (see
~\cite{EPT}).

A scheme is said to be {\it circular-arc scheme} if it is the
scheme of a circular-arc graph. In the following theorem we give a
characterization of  association circular-arc schemes.
\begin{thrm}\label{schemetheorem}
A  circular-arc scheme is association if and only if it is
isomorphic to the wreath product of a rank $2$ scheme and a
scheme which is either  forestal or dihedral.
\end{thrm}
\begin{remark}\label{remark}
Forestal scheme  in Theorem~\ref{schemetheorem} occurs as a
wreath product of the scheme of at most on 2 points and a rank 2
scheme.  For example: the scheme of a  disjoint union of copies of
a complete graph, namely $mK_n$, or the scheme of the
lexicographic product of a complete graph and a complete graph
without perfect matching.
\end{remark}

A scheme $\mathcal{X}$  is said to be  {\it Schurian} if
$\mathcal{X}=\Inv(G)$ for some permutation group $G$, see
\cite{Zi1}. In fact, any rank 2 scheme and any dihedral scheme
are Schurian. Moreover,  the wreath product of two Schurian
schemes is Schurian, see~\cite{Zi1}. Thus  we have the following
corollary:
\begin{cor}
Any  association circular-arc scheme is Schurian.
\end{cor}
It is known  that the automorphism group of each graph is equal to
automorphism group of its scheme, see~\cite{Wax}. Moreover, it is
well-known that the automorphism group of wreath product of two
schemes is equal to  the wreath product of their automorphism
groups. Denote by~$\sym(n)$ the symmetric group   on $n$ points,
and denote by $G\wr H$ the wreath product of two groups $G$ and
$H$. The following corollary is an immediate consequence of
Theorems~\ref{graphtheorem} and \ref{schemetheorem}.
\begin{cor}
Let $\Gamma$ be a  circular-arc graph with association scheme on
$n$ vertices. Then there is an even integer $k$, $k | n$,  such
that $\aut({\Gamma})$ is isomorphic to  $\sym(\frac{n}{k}) \wr G$,
where $G$ is $\sym(k)$ or $\sym(2)\wr \sym(\frac{k}{2})$ or
$D_{2k}$.
\end{cor}
\medskip

This paper is organized as follows. In Section 2, we present some
preliminaries on graph theory and  scheme theory. In Section 3,
we first remind the concept of circular-arc graphs and then we
introduce arc-function and reduced arc-function of a circular-arc
graph. Moreover,  we characterize non-empty regular circular-arc
graphs without twins.
 Section 4 contains  relationship  between lexicographic product of graphs and wreath product of their schemes.
    In Section 5, we define elementary circular-arc
 graphs. Then, we characterize the scheme of graphs which belong to this class.
 Finally, in Section 6  we give the proof of Theorem~\ref{graphtheorem} and
 Theorem~\ref{schemetheorem}.

\medskip
{\bf Notation.} Throughout the paper, $V$ denotes a finite set.
The diagonal of the Cartesian product $V^2$ is denoted by~$1_V$.
\medskip

For $r,s\subset V^2$ and $X,Y \subset V$ we have the following
notations:

 $$r^*=\{(u,v)\in V^2:\ (v,u)\in r\},$$ $$r_{X,Y}=r\cap(X\times Y) ~,
 ~r_X=r_{X,X},$$
 $$r\cdot s=\{(v,u)\in V^2:\ (v,w)\in r,\
(w,u)\in s ~{\text{ for some}}~ w\in V\},$$
$$
r\otimes s=\{((v_1,v_2),(u_1,u_2))\in V^2\times V^2:\
(v_1,u_1)\in r ~{\text{ and}} ~  (v_2,u_2)\in s\},
$$

  Also for any $v \in V$, set $vr=\{u\in V :\
(v,u)\in r \}$ and $n_r(v)=|vr|$.
\medskip

For $S\in 2^{V^2}$ denote by $S^\cup$ the set of all unions of
the elements of $S$, and set $S^*=\{s^*:\ s\in S\}$ and $v
S=\cup_{s\in S}v s$. For $T\in 2^{V^2}$ set $$S\cdot T=\{s\cdot
t:\ s\in S,\, t\in T\}.$$

 For an
integer $n$, let $\mathbb{Z}_n$ be the ring of integer numbers
modulo $n$. Set
$$
A\mathbb{Z}_n:=\{\{i, i+1,\ldots, i+k\}: i,k \in \mathbb{Z}_n
~{\text {and}}~ k \neq n-1 \}.
$$
For each set $\{i, i+1,\ldots, i+k\}$, the points $i$ and $i+k$
are called the end-points of the set.

\medskip

\section{Preliminaries }
\subsection{Graphs.}\label{graph} All terminologies and definitions about graph theory have been adapted from~\cite{bondy}.
 In this paper, we consider finite and undirected  graphs which contains  no loops and multiple edges.
We denote complete graph on $n$ vertices by $K_n$, and an
undirected cycle on $n$ vertices by~$C_n$.

\medskip
Let $\Gamma=(V,R)$ be a graph with vertex set $V$ and edge set
$R$. Let $E$ be an equivalence relation on  $V$, then
$\Gamma_{V/E}$ is a graph with vertex set $V/E$ in which
distinct  vertices $X$ and $Y$ are adjacent if and only if at
least one vertex  in $X $ is adjacent
 in $\Gamma$ with some vertex in $Y$.
For every subset $X$ of $V$, the graph $\Gamma_X$ is the subgraph
of $\Gamma$ induced by~$X$.

\medskip
Let $\Gamma_i$ be a graph on $V_i$, for $i=1,2$. The graphs
$\Gamma_1$ and $\Gamma_2$ are {\it isomorphic } if there is a
bijection $f: V_1 \to V_2$, such that two vertices $u$ and $v$ in
$V_1$ are adjacent in $\Gamma_1$ if and only if their images
$f(u)$ and $f(v)$ are adjacent in $\Gamma_2$. Such a bijection is
called an {\it isomorphism} between $\Gamma_1$ and $\Gamma_2$.
The set of all isomorphism between $\Gamma_1$ and $\Gamma_2$ is
denoted by $\iso(\Gamma_1,\Gamma_2)$. An isomorphism from a graph
 to itself is called an {\it automorphism}.
The set of all automorphisms of a graph $\Gamma$ is the {\it
automorphism group} of $\Gamma$, and  denoted by $\aut(\Gamma)$.

\medskip
The {\it lexicographic product} or {\it composition} of  graphs
$\Gamma_1$ and $\Gamma_2$ is the  graph $\Gamma_1[\Gamma_2]$ with
vertex set $V_1\times V_2$ in which $(u_1,u_2)$ is adjacent to
$(v_1,v_2)$ if and only if either $u_1$ and $v_1$ are adjacent in
$\Gamma_1$ or $u_1=v_1$ and also   $u_2$ and $v_2$ are adjacent in
$\Gamma_2$.

\medskip
Let $\Gamma=(V,R)$ be a graph.
 Two vertices $u,v \in V$ are {\it twins} if  $u$ and $v$ are adjacent in $\Gamma$  and  $vR \backslash\{u\}=uR
\backslash\{v \}$, where  the set of neighbors of a vertex $v$ in
the graph $\Gamma$ is denoted by $vR$.

\medskip

\subsection{ Schemes.}\label{scheme} In this part  all terminologies and
notations are based on~\cite{EP}.

\begin{definition} A pair $\mathcal{X}=(V,S)$, where $V$ is a
finite set and $S$ a partition of~$V^2$, is called a {\it scheme}
on $V$ if $1_V\in S^\cup$, $S^*=S$, and for any $r,s,t\in S$ the
number
$$
c_{rs}^t:=|v r\cap us^*|
$$
does not depend on the choice of $(v,u)\in t$. The scheme
$\mathcal{X}$ is called {\it association} if $1_V\in S$.
\end{definition}
 The elements of
$V$, $S$, $S^\cup$ and the numbers $c_{rs}^t$ are called the {\it
points}, the {\it basic relations}, the {\it relations} and the
{\it intersection numbers} of the scheme~$\mathcal{X}$,
respectively.
 The numbers
$|V|$ and $|S|$ are called the {\it degree} and the {\it rank} of
$\mathcal{X}$. A unique basic relation containing a pair
$(v,u)\in V^2$ is denoted by $r_{\mathcal{X}}(v,u)$ or $r(v,u)$.

\medskip
An equivalence relation on a subset of $V$ belonging to $S^\cup$
is called an equivalence relation of the scheme $\mathcal{X}$.
 Any scheme has {\it trivial} equivalence relations:
$1_V$ and $V^2$. Let $e\in S^\cup$ be an equivalence relation.
For a given $X\in V/e$  the {\it restriction} of the
scheme~$\mathcal{X}$ to~$X$  is the scheme
$$
\mathcal{X}_X=(X,S_X),
$$
where $S_X$ is  the set of all non-empty relations $r_X$ with
$r\in S$.  The {\it quotient} of the scheme~$\mathcal{X}$
modulo~$e$ is defined to be the scheme
$$
\mathcal{X}_{V/e}=(V/e,S_{V/e}),
$$
where $S_{V/e}$ is the set of all non-empty relations of the form
$\{(X,Y):\  s_{X,Y}\neq \emptyset\}$, with $s\in S$.

Two schemes $\mathcal{X}_1$ and $\mathcal{X}_2$ are called {\it
isomorphic} if there exists a bijection between their point sets
in such a way  that induces a bijection between their sets of
basic relations. Such a bijection is called an {\it isomorphism}
between $\mathcal{X}_1$ and $\mathcal{X}_2$. The set of all
isomorphism between $\mathcal{X}_1$ and $\mathcal{X}_2$ is
denoted by $\iso(\mathcal{X}_1,\mathcal{X}_2)$. The group of all
isomorphisms of a scheme $\mathcal{X}$ to itself contains a normal
subgroup
$$
\aut(\mathcal{X})=\{f\in \sym(V): s^f=s, s\in S\}
$$
called the {\it automorphism group} of $\mathcal{X}$ where
$s^f=\{(u^f,v^f):(u,v)\in s\}$.
\medskip

The {\it wreath product} $\mathcal{X}_1\wr \mathcal{X}_2$ of two
schemes $\mathcal{X}_1=(V_1,S_1)$ and $\mathcal{X}_2=(V_2,S_2)$
is a scheme on  $V_1\otimes V_2$ with the following  basic
relations
$$
V^2_1\otimes r,\   ~{\text{for}}~ r\in S_2\backslash 1_{V_2}
~~{\text{and}}~~ s\otimes 1_{V_2} ,\  ~{\text{for}}~ s\in S_1.
$$

\medskip

\subsection{ The scheme of a graph.}\label{s of g}
 There is a natural partial order\, $"\leq"
$\, on the set of all schemes on  $V$. Namely, given two schemes
$\mathcal{X}=(V,S)$ and $\mathcal{X}'=(V,S')$ we set
$$
\mathcal{X}\leq \mathcal{X}' \ \Leftrightarrow\ S^\cup\subseteq
(S')^\cup.
$$
In this case $\mathcal{X}$ is called a {\it fusion} (subscheme)
of  $\mathcal{X}'$ and $\mathcal{X}'$ is called a  {\it fission}
(extension) of $\mathcal{X}$. The minimal and maximal elements
with respect to $"\leq"$ are the {\it trivial} and the {\it
complete} schemes on $V$ respectively: the basic relations of the
former one are the reflexive relation $1_V$ and (if $|V|>1$) its
complement in $V^2$, whereas the relations of the later one are
all binary relations on $V$.
 \medskip

Let $R$ be an arbitrary relation on the set  $V$. Denote   by
$\fis(R)$ the smallest scheme   with respect to $"\leq"$ such
that  $R$ is a union of its basic relations.
 Let
$\Gamma=(V,R)$ be a graph with vertex set $V$ and edge set $R$. By
the scheme of $\Gamma$ we mean $\fis(\Gamma)=\fis(R)$. For
example, if $\Gamma$ is a complete graph or empty graph with at
least 2 vertices, then its scheme is the trivial scheme of rank
2. One can check  that   if  $\fis(\Gamma)$ is an association
scheme, then $\Gamma$ is a regular graph.

In general, it is quite difficult to find the scheme of an
arbitrary graph. In~\cite{EPT},  the scheme of a graph has been
studied for some classes of graphs.

\section{Circular-arc graphs}
From \cite{bondy}, for a given  family $ \mathcal{F}$ of subsets
of $V$, one may associate an {\it intersection graph}. This is the
graph whose vertex set is $\mathcal{F}$, two different sets in
$\mathcal{F}$ being adjacent if their intersection is non-empty.
 Circular-arc graph is the
intersection graph of a family  of arcs of a circle.
\begin{lem}\label{arcfun}
Let   $\Gamma$  be a graph on  $V$ with   $n$ vertices. Then
$\Gamma$ is a circular-arc graph if and only if there exists a
function $f:V \to A\mathbb{Z}_{m}$, for some $m$, such that
$\Gamma$  is the intersection graph of the family $\im(f)=\{f(v):
v\in V\}$. Moreover, this function can be chosen such that
\begin{enumerate}
   \item  [(1)] any element of $\mathbb{Z}_{m}$ is the end-point of at least one set
   in
  $\im(f)$,
  \item [(2)] each set  in $\im(f)$ contains at least two
  elements.
\end{enumerate}
\end{lem}
\begin{proof}
To prove sufficient part let $\Gamma$ be   a circular-arc graph.
Then by the  definition it is the intersection graph of some arcs
of a circle $C$. Without loss of generality, we may assume that
the end-points of any of these arcs are distinct. Let $m$ be the
number of these end-points and let $A=\{a_0, a_1,\ldots,
a_{m-1}\}$ be  the set of all of them. It is clear that $m\leq
2n$. Here the indices of the points $a_i$ are the elements of
$\mathbb{Z}_m$; they are chosen in such a way that the point $a_i$
precedes the point $a_{i+1}$ in the clockwise order of the circle
$C$. Then for any vertex $v\in V$ there exist uniquely determined
elements $i_v, j_v \in \mathbb{Z}_m$ such that
$$
A_v:=C_v \cap A= \{a_{i_v},a_{i_v+1}\ldots, a_{j_v}\},
$$
where $C_v$ is  a subset of $C$ which is  the arc corresponding to
$v$. Moreover, it is easily seen that $C_u \cap C_v$ is not
empty  if and only if $i_v\in A_u$ or $j_v\in A_u$ or $i_u,j_u\in
A_v$. Therefore, the vertices $u$ and $v$ are adjacent if and only
if the set $A_u \cap A_v$ is not empty. Now define a  function
$f:V \to A\mathbb{Z}_{m}$ by
$$
f(v)=\{i_v, i_v+1, \ldots, j_v\}.
$$
Then $\Gamma$  is the intersection graph of the family $\im(f)$.
Moreover, statements (1) and (2) immediately follows from the
definition of $f$.

Conversely,  let  $m\leq 2n$ and  $f:V \rightarrow
A\mathbb{Z}_{m}$ be a function such that $\Gamma$  is the
intersection graph of  $\im(f)$. Consider a circle $C$ and choose
$m$ distinct  points on it. We may label these points by the
elements of $\mathbb{Z}_m$ such that these points appears in $C$
in clockwise order.  Since $\im(f)\subset A\mathbb{Z}_m$  for each
vertex $v\in V$ there exist $i_v,j_v\in \mathbb{Z}_m$ such that
$f(v)= \{i_v, i_v+1, \ldots, j_v\}$.  We correspond an arc $C_v
\subset C$,  from $i_v$ to $j_v$ in clockwise order to the vertex
$v$. It is clear that the set $f(u)\cap f(v)$ is not empty if and
only if $C_u \cap C_v$ is not empty. It follows that $\Gamma$ is
the intersection graph of the set  $\{C_v : v\in V\}$. So it is a
circular-arc graph. This completes the proof of the lemma. \hfill
$\Box$
\end{proof}

\vspace{5mm}

\medskip

The function $f:V\to A\mathbb{Z}_m$  satisfying statements~(1)
and~(2) and conditions of Lemma~\ref{arcfun}    is called the
{\it arc-function} of the graph~$\Gamma$ and the number $m$ is
called the {\it length} of $f$.\\
\begin{thrm}\label{reduced}
Let $\Gamma=(V,R)$ be a non-empty  circular-arc graph on $n$
vertices. Suppose that for any two vertices $u$ and $v$ in $V$ we
have:
\begin{equation}\label{subset}
v\in uR ~~~ \Rightarrow~~~ uR \not\subset  \{v\}\cup vR.
\end{equation}
 Then there exists
an arc-function $f$ of $\Gamma$ such that the following statements
hold:
\begin{enumerate}
\item [(i)] no set of $\im(f)$ is a subset of another set of $\im(f)$,
   \item  [(ii)] the length of ~$f$ is equal to $n$,
  \item [(iii)] any element  $i \in \mathbb{Z}_{n}$ is the end-point of exactly two sets in
  $\im(f)$.
\end{enumerate}
\end{thrm}

\begin{remark}\label{remark1}
The graph $\Gamma$  satisfying  condition \eqref{subset} is a
connected graph. Indeed, otherwise, it is easily seen that
$\Gamma$ is an  interval graph. On the other hand,
  each interval graph is chordal, and so it  has a  vertex whose  neighborhood  is a complete graph (see~\cite{bondy}) which
   contradicts the condition \eqref{subset}.
\end{remark}
\begin{proof}
By Lemma \ref{arcfun}, there exists  an arc-function $f'$ of
$\Gamma$ of length $m'\leq 2n$. Denote by $\sim$ the binary
relation on  $\mathbb{Z}_{m'}$ defined by  $i\sim j$ if and only
if for any $ v\in V$
$$
 i,j\in f'(v) {\text{~or~}} i,j\notin f'(v).
$$
One can check  that $\sim$ is an equivalence relation, and
 its equivalence  classes  belong to $A\mathbb{Z}_{m'}$. By the definition of $"\sim"$ any
 element of $\im(f')$ is a disjoint union of some classes.
Let us define a function $f$ such that for each $v\in V$, $f(v)$
is the set of $\sim$-classes contained in $f'(v)$. The
equivalence classes of $\sim$ can be identified by
$\mathbb{Z}_{m}$. By this identification we have $f(v)\in
A\mathbb{Z}_{m}$.

 We claim that $f$ is an
arc-function of $\Gamma$.  Indeed, from the  definition of $f$ it
follows that  for each two vertices $u$ and $v$,  the set
$f(u)\cap f(v)$ is not empty if and only if the set $f'(u)\cap
f'(v)$ is not empty. Moreover, statement~(1) of
Lemma~\ref{arcfun} is obvious. Thus,  it is sufficient to verify
that  statement~(2) of Lemma \ref{arcfun} occurs. Suppose on the
contrary that there is a vertex $v\in V$ such that $f(v)$
contains exactly one element. Then $f'(v)$ is a class of the
equivalence $\sim$. By condition \eqref{subset} this implies that
$v $ is an isolated vertex in
 $\Gamma$, which is impossible by Remark \ref{remark1}. Thus $f$ is an
arc-function of $\Gamma$.

By Lemma~\ref{arcfun}, the graph $\Gamma$ is isomorphic to the
intersection graph of the family $\im(f')$. Thus for  two
adjacent vertices $u$ and $v$ in $V$, if $f'(u)\subseteq f'(v)$
then any vertex in $V\backslash \{v \}$ which is adjacent to $u$
in $\Gamma$  is also  adjacent to $v$. On the other hand, it is
easy to see that $f'(u)\subseteq f'(v)$ is equivalent to
$f(u)\subseteq f(v)$. Therefore, we have
\begin{equation}\label{vr}
f(u)\subseteq f(v) ~~~ \Rightarrow~~~ uR \subset \{v\}\cup vR.
\end{equation}
Hence,  statement ({\it i}) follows from condition~\eqref{subset}.

Statement ({\it ii}) is a consequence of statement ({\it iii}).
First we will show that any element $i\in \mathbb{Z}_{m}$ is the
end-point of exactly two sets in $\im(f)$. Suppose on the
contrary that there is an element $i\in \mathbb{Z}_{m}$ which is
an end-point of at least three sets of $\im(f)$. By statement (2)
of Lemma \ref{arcfun}, there are  at least  two sets  $f(u)$ and
$f(v)$ such that
$$
i+1\in f(v)\cap f(u)~~{\text or }~~i-1\in f(v)\cap f(u).
$$
It follows that in any case   $f(v)\subseteq f(u)$ or
$f(u)\subseteq f(v)$. From~\eqref{vr}, this  contradicts
condition~\eqref{subset}. Thus any element $i\in \mathbb{Z}_{m}$
is the end-point of at most two sets in $\im(f)$.

To complete the proof, suppose that   there exists $i\in
\mathbb{Z}_{m} $ which is an  end-point of exactly one set, say
$f(v)$, in $\im(f)$. By statement (2) of Lemma~\ref{arcfun}, we
have $i+1\in f(v)$ or $i-1\in f(v)$. Suppose that the former
inclusion holds. Then by~\eqref{vr} we have
\begin{equation}\label{uvr}
u\in vR ~  \Rightarrow ~   i \not\in f(v)\cap f(u).
\end{equation}
 If $i+1$ is an end-point of $f(v)$ then,
 since $\Gamma$ does not contain any isolated vertex,
so  there is a vertex $u\in V $ such that $f(u)\cap f(v)$ is not
empty. From~\eqref{uvr} it follows that  $i+1$ is an end-point of
$f(u)$. Let $w\in vR$, then the set  $f(v)\cap f(w)$ is not
empty. From~\eqref{uvr} we have $i+1\in f(w)$ and so $f(w)\cap
f(u)$ is not empty. So, $w\in uR$. Therefore, in this case $vR
\subseteq \{u\} \cup uR$, that contradicts the condition
\eqref{subset}. Thus we  suppose  $i+1$ is not an end-point of
$f(v)$. In this case,  statement (1) of Lemma \ref{arcfun}
implies that  $i+1$ is an end-point of a set in $\im(f)$, say
$f(u)$. From Lemma~\ref{arcfun}, we have $f(u)\nsubseteq f(v)$ and
it follows that $f(v)\backslash \{i\} \subseteq f(u)$. Now
from~\eqref{uvr}, we have $vR\subseteq \{u\} \cup uR$, which
contradicts the condition~\eqref{subset}. If $i-1\in f(v)$, by
the same argument we have a contradiction again.  Thus, any
element $i\in \mathbb{Z}_{m}$ is the end-point of exactly two
sets in $\im(f)$. This completes the proof of the theorem.
\hfill$\Box$
\end{proof}
\vspace{5mm}

\medskip
Any arc-function $f:V\to A\mathbb{Z}_n$, satisfying  conditions
({\it i}), ({\it ii}) and~({\it iii}) of Theorem~\ref{reduced} is
called the {\it reduced arc-function} of the graph~$\Gamma$.
\medskip
\begin{cor}\label{n(v)}
Let $\Gamma=(V,R)$ be a  graph which satisfies the conditions of
Theorem \ref{reduced}. Then  $n_R(v)=2|f(v)|-2$ for each vertex
$v\in V$, where $f$ is the reduced arc-function of $\Gamma$.
\end{cor}
\begin{proof}
Let $v\in V$. From statement  ({\it iii}) of Theorem
\ref{reduced} any $i\in f(v)$ is the end-point of exactly two
elements in $\im(f)$. Therefore, we get
\begin{equation}\label{nrf}
n_R(v) \leq 2|f(v)|-2.
\end{equation}
In fact,  by  statement  ({\it i}) of Theorem \ref{reduced}  for
each $u\in vR$, exactly one of the end-points of $f(u)$ belongs
to $f(v)$. Thus we have equality in~\eqref{nrf}, and we are done.
\hfill $\Box$
\end{proof}
\begin{prop}\label{regular}
Let $\Gamma=(V,R)$ be a non-empty circular-arc graph without
twins. Then $\Gamma$ is regular  if and only if for any two
vertices $u$ and $v$ in $V$ we have:
\begin{equation}\label{subset1}
v\in uR ~~~ \Rightarrow~~~ uR \not\subset  \{v\}\cup vR.
\end{equation}
\end{prop}
\begin{proof}
 Suppose that $\Gamma$ is regular and $u$ and $v$ are two adjacent vertices of the graph $\Gamma$. Then $| \{v\} \cup vR|=|\{u\}\cup uR|$.
However,  $ \{v\} \cup vR \neq \{u\}\cup uR$ because $u$ and $v$
are not twins.
 It follows that there exists a vertex in $uR$, different from  $v$, which is
not in $vR$;  and  there exists a vertex in $vR$, different from
$u$, which  is not in $uR$. Therefore, the condition
\eqref{subset1} holds.

Conversely, suppose that  $\Gamma$ satisfies
condition~\eqref{subset1}. By the same argument as  Remark
\ref{remark1} the graph $\Gamma$ is  connected. Thus, it is
sufficient to show that any two adjacent vertices $u$~and~$v$
have the same degree. On the other hand, by Theorem~\ref{reduced}
there is a reduced arc-function $f:V\to A\mathbb{Z}_n$,
 where $n=|V|$. So by Corollary~\ref{n(v)} it is
sufficient to show that $|f(v)|=|f(u)|$, or equivalently
\begin{equation}\label{feq}
|f(u)\backslash f(v)|=|f(v)\backslash f(u)|.
\end{equation}
Note that the set   $f(u)\cap f(v)$ is not empty because  $u$ and
$v$ are adjacent.
 Moreover, by hypothesis,  $u$~and $v$ are
not twins so $f(v)\neq f(u)$. We may assume that
\begin{equation}\label{cap}
 f(u)\cup f(v) \neq \mathbb{Z}_n.
\end{equation}
Indeed, otherwise, we would have $f(u)\cup f(v)=\mathbb{Z}_n$ and
then from statement~({\it i}) of Theorem~\ref{reduced}, any set in
$\im(f)$ different from both $f(u)$ and $ f(v) $ has one end-point
in $f(v)$ and one end-point in $f(u)$. This implies that  any
vertex in $V\backslash \{u,v\}$ is adjacent to both $u$ and $v$,
which is impossible because $u$ and $v$ are not twins.

 Let $i\in f(u)\backslash f(v)$. Then  by  statement
({\it iii}) of Theorem \ref{reduced},   there are exactly two
vertices $u_i, v_i\in V$, for which $i$ is the end-point
 of both
$f(u_i)$ and $f(v_i)$, or equivalently due to~\eqref{cap} we have
$$
\{i\}=f(u_i) \cap f(v_i).
$$
 Moreover,  by  statement ({\it i}) of Theorem~\ref{reduced},  neither $f(u_i)$
 nor
 $f(v_i)$ is a  subset of $f(u)$. Now, from~\eqref{cap}
 it  follows that   the  end-points of $f(u_i)$ and  $f(v_i)$ different from  $i$ is not in the
set $f(u)$. Thus,  exactly one of $f(u_i)$ or $f(v_i)$ contains
the set $f(v)\cap f(u)$. Assume that $f(u)\cap f(v)\subset
f(u_i)$. Since, by  statement ({\it i}) of
Theorem~\ref{reduced},  $f(v)$ is not a subset of $f(u_i)$, it
follows that the  end-point of $f(u_i)$, different from $i$, is
in the set $f(v)\backslash f(u)$, denote this end-point by
$j_{i}$, (see Fig. 1).

\setlength{\unitlength}{0.8cm}
\begin{picture}(6,4)
 \thicklines

\qbezier(5,0.5)(8,2)(11,0.9)

 \put(5,0.5){\circle*{0.18}}
 \put(7.15,1.22){\circle*{0.18}}
 \put(11,0.9){\circle*{0.18}}
 \put(7,0.6){$i$}
 \put(10.9,0.4){$j_i$}
 \put(8.2,0.6 ){$f(u_i)$}
 \put(5.5,0.3 ){$f(v_i)$}

\qbezier(6,1.2)(8,2)(10,1.55)
  \put(6,1.2){\circle*{0.18}}
  \put(10,1.55){\circle*{0.18}}
  \put(6.3,1.9){$f(u)$}

\qbezier(8,2)(10.5,2.2)(12.5,1)
   \put(8,2){\circle*{0.18}}
   \put(12.5,1){\circle*{0.18}}
    \put(11,1.9){$f(v)$}

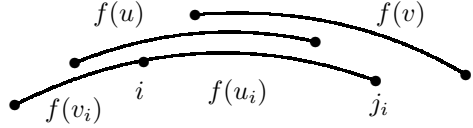
     \put(3.5,-0.8){Fig. 1:  Some arcs of reduced arc-function $f$ of $\Gamma$ }
\end{picture}
\vspace{1.3cm}

So far we could define the mapping,  $i\to j_i$, from
$f(u)\backslash f(v)$ to $f(v)\backslash f(u)$. Now we claim that
this mapping is bijection. To do so, we first prove that it is
injective.
 Suppose on the contrary that  there are
$i,i'\in f(u)\backslash f(v)$ such that  $j_{i} = j_{i'}$. Then
$f(u_i)\subset f(u_{i'})$ or $f(u_{i'})\subset f(u_{i})$.
However, in both cases this  contradicts   statement ({\it i}) of
Theorem~\ref{reduced}. Now, let $j\in f(v)\backslash f(u)$ then
in a similar way there is a corresponding element of
$f(u)\backslash f(v)$, say~$i$. By statement ({\it i}) of
Theorem~\ref{reduced}, it is clear that $j_i=j$. This shows that
the mapping is surjective.

The same argument can be done for  the
 case $f(u)\cap f(v)\subset f(v_i)$.
Thus~\eqref{feq} follows and this proves the proposition.
 \hfill $\Box$
\end{proof}
\begin{cor}\label{correduced}
Let $\Gamma$ be an $m$-regular  circular-arc graph on $n$
vertices. Suppose that $m\geq 1$ and the graph has no twins. Then
for each  reduced arc-function  $f$ of $\Gamma$ and  each $v\in
V$, we have $|f(v)|=\frac{m+2}{2}$.
\end{cor}
\begin{proof}
By the hypothesis the graph $\Gamma$ is non-empty and without
twins. Therefore, the hypothesis of Corollary \ref{n(v)} holds
and we are done. \hfill $\Box$
\end{proof}
\section{Graphs and schemes}
In this section we prove some results  on the  scheme of a graph.
In particular we will study a relationship between the
lexicographic product of two graphs and the wreath product of
their schemes.

\begin{lem}\label{rk}
Let $\Gamma=(V,R)$ be a graph. For each integer $k$, let
$$
R_k=\{(u,v) \in R :  ~~ |uR \cap vR|=k \}.
$$
Then $R_k$ is a union of some basic relations of the scheme
$\fis(\Gamma)$.
\end{lem}
\begin{proof}
Let $S$ be the set of basic relations of $\fis(\Gamma)$. Then
$R=\bigcup_{s\in S'} s$ where $S' \subset S$.  It is sufficient
to show that  $R_k$ contains any relation from $S'$ whose
intersection with $R_k$ is not empty. To do this, let $s$ be such
a relation. Then there exists a pair $(u,v)\in s$ such that  $|uR
\cap vR|=k$. On the other hand, by the definition of intersection
numbers we have
$$|uR \cap vR|= \sum_{r,t\in S'}c^s_{rt}.$$  Thus the number
$|uR \cap vR|$ does not depend on the choice of $(u,v)\in s$. By
definition of $R_k$ this implies that $s \subset R_k$ as required.
\hfill $\Box$
\end{proof}

\medskip
\begin{thrm}\label{equ}
Let $\Gamma$ be a    graph on the vertex set $V$ and let
$$
E=\{(u,v)\in V\times V~ :~ u ~{\rm{and}}~ v ~{\rm{are ~twins
~or}}~ u=v\}.
$$
 Then $E$ is an  equivalence relation  of the scheme $\fis(\Gamma)$. Moreover, if $\Gamma$
is a graph with association scheme then, it  is isomorphic to
lexicographic product of the graph $\Gamma_{V/E}$ and  a complete
graph.
\end{thrm}
\begin{proof}
Let $S$ be the set of basic relations of the scheme $\fis(\Gamma)$
and let $R$  be the edge set of $\Gamma$. Then there exists a
subset $S'$ of $S$ such that
\begin{equation}\label{r}
 R=\bigcup_{s\in S'} s.
 \end{equation}
To prove the first statement, we need to check that any
non-reflexive basic relation $r\in S$ such that $r\cap E\neq
\emptyset$ is contained in $E$. To do this, let $(u,v)\in r$. We
claim that  $(u,v)\in E$, or equivalently $u$ and $v$ are twins.

First we show that
\begin{equation}\label{sub}
uR\backslash \{v\} \subseteq vR\backslash \{u\}.
\end{equation}
 If the set
$uR\backslash \{v\}$ is empty, then  \eqref{sub} is clear. Now, we
may assume that there exists an element  $w$ in $V$ such that
$w\in uR\backslash \{v\}$. It is enough to show that $v$ is
adjacent to $w$ in $\Gamma$. By \eqref{r} there exists a basic
relation $s\in S'$ so that $(u,w)\in s$. Denote by $t$ the basic
relation which contains $(v,w)$. It is sufficient to show that
$t\in S'$.

We have $w\in us\cap vt^*$, thus  $|us\cap vt^*|=c^r_{st^*}\neq
0$. Since intersection number does not depend on the choice of
$(u,v)\in r$,  for   $(u',v')\in r$ we have
\begin{equation}\label{crst}
|u's\cap v't^*|=c^r_{st^*}\neq 0.
 \end{equation}
 On the other hand, by the choice of
$r$ there exists $(u',v')\in r\cap E$. So by~\eqref{crst}, there
exists a vertex $w'\in V$ such that
\begin{equation}\label{prim}
w'\in u's\cap v't^*.
\end{equation}
 Moreover, since $s\in S'$, we have $w'\in u'R\backslash \{v'\}$.
 On the other hand, $u'R\backslash \{v'\}
=v'R\backslash \{u'\}$, since $u'$ and $v'$ are twins.  It
follows that $w'\in v'R\backslash \{u'\}$, so from \eqref{prim}
we conclude that $w'$ is adjacent to $v'$ in $\Gamma$ and so  from
\eqref{r} we have $t\in S'$.  The converse inclusion of
\eqref{sub} can be proved in a similar way. Thus $u$ and $v$ are
twins and the first statement  follows.

To prove the second statement, suppose that  $\Gamma$ is a graph
with association scheme. It is well-known fact that all classes of
an equivalence relation of an association scheme have the same
size, say $m$, where $m$ divides $n=|V|$. Moreover, for each
$X\in V/E$ we have $u, v \in X $ if and only if $u$ and $v$ are
twins. Thus for each $X\in V/E$ we have
\begin{equation}\label{twin}
\Gamma_X \simeq K_m.
\end{equation}
 Fix a class
$X_0 \in V/E$.  For each $X\in V/E$ choose an isomorphism  $f_X
\in \iso(\Gamma_{X_0},\Gamma_{X})$. Then the mapping
\begin{equation}\label{iso1}
 f:V \to V/E \times X_0 \\
\end{equation}
$$~~~~~~~v \mapsto (X, f^{-1}_{X}(v)),$$
is a bijection, where $X$ is a class of $E$  containing  $v$. In
order to  complete the proof, we show that the above  bijection is
a required isomorphism:
\begin{equation}\label{iso}
f\in \iso(\Gamma, \Gamma_{V/E}[\Gamma_{X_0}]).
\end{equation}
Take two different vertices  $u$ and $v$ in $V$, then
$f(u)=(X,u_0)$ and $f(v)=(Y, v_0)$, where $X,Y\in V/E$, $u\in X$,
$v\in Y$ and $u_0, v_0 \in X_0$. It is enough to show that $u$
and $v$ are adjacent in $\Gamma$ if and only if $f(u)$ and $f(v)$
are adjacent in $\Gamma_{V/E}[\Gamma_{X_0}]$.

First, we assume that $u$ and $v$ are not twins. Then  $X \neq Y$.
In this case,   by definition of $E$, if  $u$ and $v$ are
adjacent in $\Gamma$ then any vertices in $X$ and any vertices in
$Y$ are adjacent to each other. Also if $X$ and $Y$ be adjacent in
$\Gamma_{V/E}$, by definition of $\Gamma_{V/E}$, there is a
vertex in $X$ and a vertex in $Y$ which are adjacent. However,
since all of the vertices in each class of $V/E$ are twins, then
$u$ and $v$ are adjacent in $\Gamma$. Thus $u$ and $v$ are
adjacent in $\Gamma$ if and only if  $X$ and $Y$ are adjacent in
$\Gamma_{V/E}$. Therefore, $u$ and $v$ are adjacent in $\Gamma$
if and only if  $f(u)$ and $f(v)$ are adjacent in
$\Gamma_{V/E}[\Gamma_{X_0}]$.

 Now, we may assume that   $u$ and $v$ are twins. Then $X=Y$. However, $u$ and $v$ are twins, so they are adjacent
in $\Gamma$. Since, $f_X$ is an isomorphism thus
$f^{-1}_{X}(u)\neq f^{-1}_{X}(v)$, so $u_0\neq v_0$.
From~\eqref{twin}, it follows that   $u_0$ and $v_0$ are adjacent
in $\Gamma_X$.
 Therefore, $f(u)$ and $f(v)$ are adjacent in
$\Gamma_{V/E}[\Gamma_{X_0}]$. Thus   $f$ is an isomorphism  and
\eqref{iso} follows, as desired. \hfill $\Box$
\end{proof}

\medskip
\vspace{5mm}

\medskip
In the next theorem we show that  the scheme of the lexicographic
product of two graphs is smaller than the wreath product of their
schemes. In general, we do not have equality here. For example,
the scheme of the lexicographic product of two complete graphs is
a scheme of rank 2, but the wreath product of their  schemes has
rank 3.

\begin{thrm}\label{wreath}
Let $\Gamma_1$ and $\Gamma_2$ be two graphs. Then the scheme
$\mathcal{X}=\fis(\Gamma_2[\Gamma_1])$  is isomorphic to a fusion
of the scheme $\mathcal{Y}= \fis(\Gamma_1) \wr  \fis(\Gamma_2)$.
 Moreover, if $\Gamma_1$ is a complete graph and
$\Gamma_2$ is a graph   without twins such that its scheme is
association, then $\mathcal{X}$ and $\mathcal{Y}$ are isomorphic.
\end{thrm}
\begin{proof}
Let $\Gamma_i=(V_i,R_i)$, $\mathcal{X}_i=\fis(\Gamma_i)$ and $S_i$
be the set of basic relations of $\mathcal{X}_i$ for $i=1,2$.
Then  there exists $S'_i\subset S_i$ such that
\begin{equation}\label{ri}
R_i=\bigcup_{s\in S'_i}s.
\end{equation}
Let $\Gamma$ be the lexicographic product of $\Gamma_2$ and
$\Gamma_1$, and let  $R$ be the edge set of  $\Gamma$. Then we
have
$$R=\{((i,k),(j,l)) \in  (V_2\times V_1)^2 :~  (i,j)\in R_2 ~~{\text {or}}~~~  (k,l)\in R_1  ~{\text {with}}~ i=j  \}.$$
Let  $\Gamma'$ be a graph on vertex set $V_1\times V_2$ such that
$(k,i)$ and $(l,j)$ are adjacent in $\Gamma'$ if and only if
$(i,k)$ and $(j,l)$ are adjacent in $\Gamma$. Define
$\sigma:V_2\times V_1 \to V_1\times V_2 $ such that
$(i,k)^{\sigma}=(k,i)$. Then $\sigma \in \iso(\Gamma,\Gamma')$.
Thus $\Gamma$ and $\Gamma'$ are isomorphic and it follows that
\begin{equation}\label{fis}
\fis(\Gamma)^{\sigma}= \fis(\Gamma').
\end{equation}
Let $R'$ be the edge set of  $\Gamma'$, then
$$R'=\{((k,i),(l,j)) \in  (V_1\times V_2)^2 :  (i,j)\in R_2 ~~{\text {or}}~~~ (k,l)\in
R_1 ~{\text {with}}~ i=j \}$$
$$=\{((k,i),(l,j)) \in  (V_1\times V_2)^2 :  (i,j)\in R_2 \}\cup~~~~~~~~~~~~~~~~~~~~~~~~~~~~$$
$$ \{((k,i),(l,j)) \in  (V_1\times V_2)^2 :  (k,l)\in R_1  ~{\text {with}}~ i=j \}.~~~~~~~~~~~~~$$
So, by \eqref{ri} we have
\begin{equation}\label{r'}
R'=\{ (V_1)^2 \otimes s : s\in S'_2\}~\cup ~\{ s \otimes 1_{V_2}
: s \in S'_1\}.
\end{equation}
Therefore, $R'$ is a union of some basic relations of the scheme
$\mathcal{Y}=\mathcal{X}_1\wr \mathcal{X}_2$. Thus we conclude
that
\begin{equation}\label{fis3}
\fis(\Gamma') \leq \mathcal{Y},
\end{equation}
 and from \eqref{fis}  the first statement follows.

To prove the second statement, let $\Gamma_1$ be a complete graph
on $n$ vertices and let $\Gamma_2$ be a  graph without twins such
that $\mathcal{X}_2$ is an  association scheme. Then,
$\mathcal{X}_1\wr \mathcal{X}_2$ is association  and from the
first statement it follows that $\fis(\Gamma')$ is association.
Thus $1_{V_1}\otimes 1_{V_2}$ is a basic relation of
$\fis(\Gamma')$.

 If $\Gamma_2$ be an empty graph, then it is easy to
see that we have equality in \eqref{fis3}. So we may suppose that
$\Gamma_2$ is a non-empty graph. Since   it is a graph with
association scheme,  there exists a positive integer $t$ such
that $\Gamma_2$ is a
 $t$-regular graph.   Let $t_0(i,j)$
be the number of common neighbors of two adjacent vertices $i$ and
$j$ in $\Gamma_2$. Since $\Gamma_2$ is without twins, we have
\begin{equation}\label{d0}
t_0(i,j) < t-1.
\end{equation}
Let  $u=(k,i)$ and $v=(l,j)$ be two adjacent vertices in
$\Gamma'$. Since, for each $i\in V_2$ the graph
$\Gamma_{V_1\times i}$ is isomorphic to $\Gamma_1$, and  for each
two adjacent vertices $i$ and $j$ in $\Gamma_2$ the set
$(V_1\times i)\times (V_1\times j)$ is a subset of $R'$, thus we
have
\begin{equation}\label{delta}
|uR'\cap vR'|=\left\{%
\begin{array}{lll}
     (n-2) +tn,  ~~~~~~~~~~~~~~~~~~~~i=j \\[0.2cm]
     2(n-1)+t_0(i,j)n, ~~~~~~~~~~~ i\neq j. \\[0.1cm]
\end{array}%
\right.
\end{equation}
Using \eqref{d0},  for each $i\neq j$ we have
\begin{equation}\label{delta1}
 2(n-1)+t_0(i,j)n< (n-2) +tn.
\end{equation}
Define
$$
E:=\{(u,v)\in R': |uR'\cap vR'|= (n-2) +tn\}.
$$
From~\eqref{delta} and \eqref{delta1} it  follows that
$$
E=\cup_{i\in V_2}(V_1\times i)^2\backslash ( 1_{V_1}\otimes
1_{V_2}).
$$
Now, from Lemma \ref{rk}, the set $E$  is  a union of some of the
basic relations of $\fis(\Gamma')$. On the other hand, $E=s
\otimes 1_{V_2}$, where $s$ is the non-reflexive basic relation
of the scheme $\mathcal{X}_1$ of rank 2. However, $s \otimes
1_{V_2}$ is a basic relation of the scheme  $\mathcal{X}_1\wr
\mathcal{X}_2$. Therefore, from \eqref{fis3} it is obvious  that
$E$ is a basic relation of $\fis(\Gamma')$.
 Hence,
\begin{equation}\label{f}
F=E\cup( 1_{V_1}\otimes 1_{V_2})
\end{equation}
is an equivalence relation of the scheme $\fis(\Gamma')$.

 The scheme
$\mathcal{X}_1 \wr \mathcal{X}_2$ is the minimal scheme which
contains an equivalence $F$ such that for each class $X\in V/F$,
the scheme $(\mathcal{X}_1 \wr \mathcal{X}_2)_X$ is isomorphic to
$\mathcal{X}_1$ and $(\mathcal{X}_1 \wr \mathcal{X}_2)_{V/F}$ is
isomorphic to $\mathcal{X}_2$. In order  to prove equality in
\eqref{fis3}, it is sufficient to show that $\fis(\Gamma')$ have
the above property.

Let $X\in V/F$. Then by \eqref{f}, the scheme $\fis(\Gamma')_X$ is
isomorphic to the scheme $\mathcal{X}_1$. Moreover, from
\eqref{fis3} it follows that
\begin{equation}\label{ve}
\fis(\Gamma')_{V/F}\leq \mathcal{X}_2.
\end{equation}
However, the edge set of $\Gamma_2$ is a union of some of the
basic relations of $\fis(\Gamma')_{V/F}$. Thus we have equality in
\eqref{ve}, and we are done.
 \hfill $\Box$
\end{proof}
\vspace{5mm}
\begin{remark}\label{subg}
Let $\Gamma$ and $\Gamma'$ be  two graphs  with the same  vertex
set. Suppose that the edge set of $\Gamma'$ is a union of some
basic relations of the scheme of $\Gamma$. Then $\fis(\Gamma')
\leq \fis(\Gamma)$.
\end{remark}

\section{Elementary circular-arc graphs}\label{cmn graph}

 Given integers $n$ and $k$ such that  $0\leq
2k+1< n$, set $C_{n,k}=\cay(\mathbb{Z}_n,S)$ where  $S=\{\pm
1,\ldots, \pm k\}$.  It immediately follows  that $C_{n,k}$ is a
$2k$-regular graph without twins. Note that  $C_{n,0}$ is  an
empty graph  and $C_{n,1}$ is an undirected cycle on $n$ vertices.

From definition, one can verify that $C_{n,k}$ is the graph with
vertex set $V=\mathbb{Z}_{n} $ in which  two vertices $i$ and $j$
are adjacent if and only if
$$
\{i,\ldots, k+i\}\cap \{j,\ldots, k+j\} \neq \emptyset.
$$
Suppose  that $f:V\to A\mathbb{Z}_{n}$, such that
$f(i)=\{i,\ldots,
 k+i\}$. Then the graph  $C_{n,k}$ is the intersection graph of
 the family $\im(f)$. Thus, by Lemma \ref{arcfun} we conclude that  $C_{n,k}$ is a circular-arc graph,
 and  we call it {\it elementary circular-arc graph}.

\begin{example}\label{ex4}
For $n=2k+2$, two different  vertices $i$ and $j$ are adjacent in
$C_{n,k}$ if and only if $j\neq i+k+1$. Thus  $C_{n,k}$ is
isomorphic to a graph on $n$ vertices which is obtained from a
complete graph by removing the edges of a perfect matching. It is
easy to check that the scheme of this graph is isomorphic to
$\mathcal{X}_1\wr \mathcal{X}_2$, where $\mathcal{X}_1$ is a rank
2 scheme on 2 points and $\mathcal{X}_2$ is a rank 2 scheme on
$k+1$ points.
\end{example}
\begin{thrm}\label{regularcag}
A regular circular-arc graph without twins is elementary.
\end{thrm}
\begin{proof}
 Let $\Gamma$ be a circular-arc graph with the vertex set $V$ where $n=|V|$.
 Suppose that it has no twins. Then by
Proposition~\ref{regular} and  Theorem~\ref{reduced}, there exists
a reduced arc-function $f$ of $\Gamma$, such that for each vertex
$v\in V$ we have $f(v)=\{i_v,\ldots, j_v\} $. Define a bijection
from $ V$ to $\mathbb{Z}_n$, the vertex set of $C_{n,k}$, such
that~$v\to i_v$. From Corollary~\ref{n(v)} we conclude that
$\Gamma$ is a $2k$-regular graph. Then $k=|f(v)|-1$ for any
vertex $v$ by Corollary~\ref{correduced}. Hence $j_v= i_v+ k$. By
Lemma~\ref{arcfun}, two vertices $u$ and $v$ in $V$ are adjacent
if and only if $f(u)\cap f(v)\neq \emptyset$. It follows that
$i_u$ and $i_v$ are adjacent if and only if $\{i_u,\ldots,
k+i_u\}\cap \{i_v,\ldots, k+i_v\}\neq \emptyset$. Therefore, the
bijection defined above gives the  required isomorphism.
 \hfill $\Box$
\end{proof}
\begin{thrm}\label{dihedral}
 Let $n$ and $k$ be two  positive integers such that  $2k+2 < n$.
  Then the scheme of the graph $C_{n,k}$
is isomorphic to a dihedral scheme.
\end{thrm}
\begin{proof}
 Let $R$ be the
edge set of $C_{n,k}$. For two vertices $i,j\in\mathbb{Z}_n$, we
define $d(i,j)$ be the distance of $i$ and $j$ in the graph
$C_{n,1}$. Suppose that  $i$ and $j$ be two adjacent vertices in
$C_{n,k}$. Then by definition of $C_{n,k}$, we have $d(i,j)\leq
k$. Without loss of generality we may assume that the vertices
$\{i+1,i+2,\ldots, j-1\}$ are between $i$ and $j$. Thus they are
adjacent to both of $i$ and $j$ in the graph $C_{n,k}$. Moreover,
the vertices $\{j-k,j-k+1\ldots, i-1\}$ and $ \{j+1,j+2,\ldots,
i+k\}$ are adjacent to both of $i$ and $j$ too. Thus the latter
three sets are subsets of  $iR\cap jR$, which are of size
$d(i,j)-1$, $k-d(i,j)$ and $k-d(i,j)$ respectively. Moreover,
since $n>2k+2$, they are disjoint. On the other hand,
$B_i:=\{i-k,i-k+1, \ldots , j-k-1 \}$ is the set of all other
vertices which are adjacent to $i$, and $B_j:=\{i+k+1,i+k+2,
\ldots , j+k \}$ is the set of all other vertices which are
adjacent to $j$, (see Fig. 2).

 \vspace{8mm}

\setlength{\unitlength}{0.8cm}
\begin{picture}(6,4)
 \thicklines

 \qbezier(10.9,2.0)(10.9,2.787)(10.3435,3.3435)
  \qbezier(10.3435,3.3435)(9.787,3.9)(9.0,3.9)
  \qbezier(9.0,3.9)(8.213,3.9)(7.6565,3.3435)
  \qbezier(7.6565,3.3435)(7.1,2.787)(7.1,2.0)
  \qbezier(7.1,2.0)(7.1,1.213)(7.6565,0.6565)
  \qbezier(7.6565,0.6565)(8.213,0.1)(9.0,0.1)
  \qbezier(9.0,0.1)(9.787,0.1)(10.3435,0.6565)
  \qbezier(10.3435,0.6565)(10.9,1.213)(10.9,2.0)

 \put(7.7,3.4){\circle*{0.18}}
 \put(7.3,3.6){$i$}

 \put(10.4,3.3){\circle*{0.18}}
 \put(10.6,3.6){$j$}

 \put(7.2,1.4){\circle*{0.18}}
 \put(5.8,1.3){$j-k$}

 \put(7.42,0.9){\circle*{0.18}}
 \put(5,0.6){$j-k-1$}

 \put(8.4,0.2){\circle*{0.18}}
 \put(7.5,-0.3){$j+k$}
 \put(9.7,0.2){\circle*{0.18}}
 \put(9.8,-0.3){$i-k$}

\put(10.8,1.4){\circle*{0.18}}
 \put(11.2,1.3){$i+k$}

\put(10.58,0.9){\circle*{0.18}}
 \put(11,0.6){$i+k+1$}

\qbezier(7.6,1)(8.5,0)(9.6,0.4)
  \put(7.9,0.85){$B_i$}

\qbezier(8.4,0.6)(9.5,0.2)(10.3,1.1)
 \put(9.6,1.1){$B_j$}

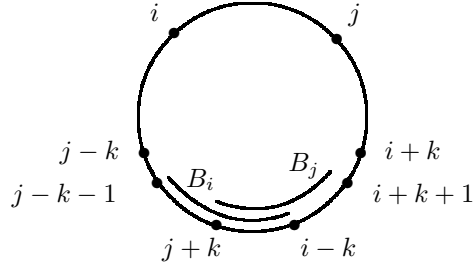
     \put(1.5,-1.5){Fig. 2:  Some vertices  of the graph $C_{n,1}$ and the sets $B_i$ and $B_j$}
\end{picture}
\vspace{2cm}

It is clear  that $B_i$ and $B_j$ are disjoint from the above
three subsets. In addition, we have $|B_i|=|B_j|=d(i,j)$.
Moreover, since $n>2k+2$, the vertex $j-k-1$ is not in $B_j$ and
the vertex $i+k+1$ is  not in $B_i$. It follows that $B_i\neq
B_j$, and
\begin{equation}\label{bij}
|B_i\cap B_j|<d(i,j).
\end{equation}
 On the other hand, $B_i\cap B_j\subset iR\cap jR$ and thus
\begin{equation}\label{ir}
~~|iR\cap jR|= (d(i,j)-1)+2(k-d(i,j))+|B_i\cap B_j|
\end{equation}
$$
= 2k-d(i,j)-1+|B_i\cap B_j|.
$$
Now, set
$$
R_{2k-2}:=\{(i,j) \in R : ~~ |iR \cap jR|=2k-2 \}.
$$
Then  $R_{2k-2}$ is  a symmetric relation. Moreover, from
\eqref{ir} we see that
\begin{equation}\label{mod1}
(i,j)\in R_{2k-2} ~\Leftrightarrow~ |B_i\cap B_j|=d(i,j)-1.
\end{equation}
If $d(i,j)=1$, then by \eqref{bij}  we have $|B_i\cap B_j|=0$.
Thus, from \eqref{mod1} we see that
\begin{equation}\label{dij1}
d(i,j)=1 ~\Rightarrow ~(i,j)\in R_{2k-2}.
\end{equation}
If $1<d(i,j)\leq k$, then  $j-k-2\in B_i\backslash B_j$ and
$i+k+2\in B_j\backslash B_i$. It follows that $|B_i\cap
B_j|<d(i,j)-1$. Thus, from~\eqref{mod1} we have $(i,j)\notin
R_{2k-2}$. Then using~\eqref{dij1} we have $(i,j)\in R_{2k-2}$ if
and only if $d(i,j)=1$.

It follows that  the graph $(\mathbb{Z}_n, R_{2k-2})$
  is isomorphic to an undirected cycle on $n$ points, say
$C_n$.  By Remark~\ref{rk}, we conclude that $R_{2k-2}$ is
 union of some basic relations of $\fis(C_{n,k})$. Thus by
Lemma \ref{subg}, we have
$$
\fis(C_n) \leq \fis(C_{n,k}).
$$
It is well-known that  $\fis(C_n)=\Inv(D_{2n})$, where $D_{2n}$ is
a dihedral group on $n$ elements. So, in order  to complete the
proof of the theorem  it is enough to show that
$$
\fis(C_{n,k})\leq \fis(C_n).
$$
 Equivalently, it is
suffices to verify that
$$
\aut(\fis(C_n)) \leq \aut(\fis(C_{n,k})).
$$
 Since the automorphism
group of a graph is equal to  the automorphism group of its
scheme,  it is sufficient to show that $\aut(C_n) \leq
\aut(C_{n,k})$.

Since $\aut(C_n)$ is the dihedral group $D_{2n}$, and $D_{2n}$ is
generated by automorphisms $\sigma$ and $ \delta$ where
$i^{\sigma}=i+1$ and $i^{\delta}=n-i$ for each $i\in
\mathbb{Z}_n$. It is enough to show  that $$\sigma, \delta\in
\aut(C_{n,k}).$$ Let $i,j \in \mathbb{Z}_n$. Two vertices $i$ and
$j$ are adjacent in $C_{n,k}$ if and only if
\begin{equation*}
\{i,\ldots, k+i\}\cap \{j,\ldots, k+j\} \neq \emptyset
\Leftrightarrow
 \end{equation*}
\begin{equation*}
\{i+1,\ldots, k+i+1\}\cap \{j+1,\ldots, k+j+1\} \neq \emptyset
\Leftrightarrow
 \end{equation*}
\begin{equation}\label{sigma}
\{i^{\sigma},\ldots, k+i^{\sigma}\}\cap \{j^{\sigma},\ldots,
k+j^{\sigma}\} \neq \emptyset.
 \end{equation}
Moreover, we have \eqref{sigma} if and only if $i^{\sigma}$ and
$j^{\sigma}$ are adjacent in $C_{n,k}$. Thus  $\sigma \in
\aut(C_{n,k})$. In a similar way we can show that $\delta \in
\aut(C_{n,k})$, this completes the proof. \hfill $\Box$
\end{proof}

\section{Proof of the main theorems}
\subsection*{Proof of Theorem \ref{graphtheorem}.}

We first prove the necessity condition of the theorem. Let
$\Gamma=(V,R)$ be a circular-arc graph with association scheme
and $|V|=n$. Denote by $E$ the equivalence relation on $V$
defined in Theorem~\ref{equ}. Then by this theorem the graph
$\Gamma$ is  isomorphic to lexicographic product of the graph
$\Gamma_{V/E}$ and  a complete graph. In particular, it is easy
to see that $\Gamma_{V/E}$ is also a circular-arc graph and has
no twins. So, to complete the proof it is enough to show that
$\Gamma_{V/E}$ is an elementary circular-arc graph.

 If $\Gamma_{V/E}$ is  empty, then it   is   isomorphic to $C_{m,0}$ with
 $m=|V/E|$, and we are done.
 We suppose that $\Gamma_{V/E}$ is non-empty.
The graph $\Gamma$ is regular, because it is a graph with
association scheme.
 By the definition of $E$ this implies that the
graph $\Gamma_{V/E}$ is regular too.  From Corollary
\ref{correduced}, the graph $\Gamma_{V/E}$ is $2k$-regular  for
some  integer $k> 0$. Therefore, from Theorem \ref{regularcag} the
latter graph is isomorphic to the elementary circular-arc graph
$C_{m,k}$. Thus $\Gamma$ is isomorphic to lexicographic product
of  an elementary circular-arc graph and a complete graph.

Conversely, let $\Gamma_1$ be a complete graph and let $\Gamma_2$
be an elementary circular-arc graph. If $\Gamma_2$ be an empty
graph then it is a graph with  association scheme. If it is a
non-empty elementary circular-arc graph then from Example
\ref{ex4} and Theorem~\ref{dihedral} we conclude that $\Gamma_2$
is a graph with association scheme. On the other hand, the wreath
product of two association scheme is association. Thus, from
 Theorem~\ref{wreath}  the scheme of $\Gamma_2[\Gamma_1]$ is
 association. This completes the proof of the theorem.
 \hfill $\Box$

\subsection*{Proof of Theorem \ref{schemetheorem}.}
We first assume that $\mathcal{X}$ is an association circular-arc
scheme. Then there is a circular-arc graph $\Gamma$ such that
$\fis(\Gamma)=\mathcal{X}$. From Theorem \ref{graphtheorem}, the
graph  $\Gamma$  is isomorphic to the  lexicographic product of an
elementary circular-arc graph and a complete graph. On the other
hand,  if the elementary circular-arc graph is empty then its
scheme is of rank 2. Otherwise, from Example~\ref{ex4} and
Theorem~\ref{dihedral}
 the  scheme of an elementary circular-arc graph
 is isomorphic to the
wreath product of a rank 2 scheme on 2 points and a rank 2
scheme, or    it is isomorphic to a dihedral scheme. Therefore,
in any case the scheme of an elementary circular-arc graph is
association. Note that in the first two cases the scheme of an
elementary circular-arc graph is  forestal. Moreover, any
elementary circular-arc graph is  without twins. Hence, from
Theorem~\ref{wreath} it follows that $\fis(\Gamma)$ is isomorphic
to  the wreath product of a  rank 2 scheme and the scheme of an
elementary circular-arc graph which is either  forestal or
dihedral.

Conversely, assume that $\mathcal{X}$ is a circular-arc scheme
such that it is isomorphic to the wreath product of a rank 2
scheme and a scheme which is either forestal or dihedral. Since
any rank 2 scheme, any  forestal scheme and any dihedral scheme
are association,  it is enough to note that
 the wreath product of two association schemes is an
association scheme.  Thus $\mathcal{X}$ is an association scheme
and  the proof is complete. \hfill $\Box$

\vspace{3mm}

\hspace{-6mm}{\bf{Acknowledgment}}
\vspace{2mm}
\par The authors are extremely grateful  to Professor Ilia Ponomarenko
for his useful discussions and valuable helps. The first author
was visiting the Euler Institute of Mathematics, St. Petersburg,
Russia during the time  this paper was written and she thanks the
Euler Institute for its hospitality.

\end{document}